\documentclass[a4paper]{scrartcl}

\usepackage[utf8]{inputenc}
\usepackage[T1]{fontenc}
\usepackage{lmodern}
\usepackage[svgnames,dvipsnames,rgb]{xcolor} 
\usepackage{amsmath,amssymb,amsthm}
\usepackage{hyperref}
\hypersetup{linktocpage=true,colorlinks=true,linkcolor=RoyalBlue,citecolor=PineGreen,urlcolor=RoyalBlue}
\usepackage{enumitem}
\usepackage[nameinlink]{cleveref}
\usepackage[margin=1in]{geometry}

\newtheorem{theorem}{Theorem}[section]
\newtheorem{lemma}[theorem]{Lemma}
\newtheorem{proposition}[theorem]{Proposition}
\newtheorem{corollary}[theorem]{Corollary}
\newtheorem{claim}[theorem]{Claim}
\theoremstyle{definition}

\newcommand{\blue}{\text{blue}}
\newcommand{\red}{\text{red}}

\newcommand*{\bfrac}[2]{\genfrac{(}{)}{}{}{#1}{#2}}

\title{Sharp thresholds for NAC-colourings and stable cuts in random graphs}
\author{Katie Clinch\thanks{School of Mathematics and Physics, University of Queensland, Australia. \texttt{k.clinch@uq.edu.au}} \and 
	John Haslegrave\thanks{School of Mathematical Sciences, Lancaster University, UK, \texttt{j.haslegrave@lancaster.ac.uk}} \and 
	Tony Huynh\thanks{Discrete Mathematics Group, Institute for Basic Science (IBS), Daejeon, South Korea, \texttt{tony@ibs.re.kr}} \and 
	Anthony Nixon\thanks{School of Mathematical Sciences, Lancaster University, UK, \texttt{a.nixon@lancaster.ac.uk}}}
\date{}

\begin{document}
	
	\maketitle
	
	\begin{abstract}
		NAC-colourings of graphs correspond to flexible quasi-injective realisations in $\mathbb {R} ^2$. A special class of NAC-colourings are those that arise from stable cuts. We give sharp thresholds for the random graph to have no stable cut and to have no NAC-colouring via exact hitting-time results: with high probability, the random graph process gains both properties at the precise time that every vertex is in a triangle. Our thresholds complement recent results on the thresholds for the random graph to be generically or globally rigid in $\mathbb {R} ^d$, and for all injective realisations to be globally rigid in $\mathbb {R} $.
	\end{abstract}
	
	\section{Introduction}
	Given a graph $G$, when is a realisation of $G$ rigid or flexible in space?  This problem has a rich mathematical history, dating at least back to the work of Cauchy and Euler who considered the rigidity and flexibility of polyhedra in $\mathbb{R}^3$. To make the above question precise, we now briefly introduce some definitions from rigidity theory.  For a more thorough introduction, we refer the interested reader to~\cite{whiteley97}.
	
	A \emph{($d$-)realisation} of a graph $G$ is a map $p: V(G) \to \mathbb{R}^d$. A realisation is \emph{generic} if its image is algebraically independent over $\mathbb{Q}$.  We say that $(G,p)$ is \emph{$d$-rigid} if the only continuous motions that preserve all edge lengths of $G$ are the Euclidean isometries of $\mathbb{R}^d$. We say that $(G,p)$ is \emph{globally $d$-rigid} if for every $d$-realisation $q:V(G) \to \mathbb{R}^d $ such that $\lVert p(u)-p(v) \rVert = \lVert q(u)-q(v) \rVert$ for all $uv \in E(G)$ we have $\lVert p(x)-p(y) \rVert = \lVert q(x)-q(y) \rVert$ for all $x,y \in V(G)$. It is easy to see that global rigidity implies rigidity, but that the converse does not necessarily hold.  
	
	A graph $G$ is \emph{generically (globally) $d$-rigid} if all of its generic $d$-realisations are (globally) rigid. It turns out that two generic $d$-realisations of a given graph will always be both (globally) rigid or both not (globally) rigid (see~\cite{AsimowRoth, connelly05, GHT10}).  Thus, generic $d$-rigidity and generic global $d$-rigidity only depend on the combinatorial structure of the graph $G$.
	
	There has been considerable recent interest in obtaining thresholds for rigidity properties in the binomial random graph. This is a very classical line of research. Indeed, the very first results in random graphs \cite{ER59, Gil59} concerned the threshold for the graph to be generically rigid in $\mathbb R$, that is, connected. The corresponding condition for a graph to be generically globally rigid in $\mathbb R$ is that it is $2$-connected, and the threshold for this was established by Erd\H{o}s and R\'enyi shortly afterwards \cite{ER61}.
	
	In higher dimensions, the relationship between connectivity and rigidity is not so simple, and the threshold for generic (global) rigidity in $\mathbb R^d$ was only determined very recently by Lew, Nevo, Peled, and Raz \cite{LNPR}. We remark, however, that these thresholds now follow more easily from Vill\'anyi's stunning result \cite{Vil25} that $d(d+1)$-connectivity is sufficient for a graph to be generically globally rigid in $\mathbb R^d$, together with standard facts about the $k$-core of a random graph.
	
	For low dimensions, it is unclear why one should restrict to generic realisations.  Indeed, Benjamini and Tzalik \cite{BT22} asked the following question. Suppose a set of $n$ distinct points in $\mathbb R$ are given, and we are given the distances between each pair independently with probability $p$. How large does $p$ need to be for us to recover the whole set up to isometry with high probability? If the point set was restricted to be generic, this would be the $2$-connectivity result of Erd\H{o}s and R\'enyi, but it is far from clear how to deal with the possibility that a non-generic realisation may be harder to reconstruct. Note, however, that we do not need every injective realisation of the random graph corresponding to revealed distances to be globally rigid, since the realisation is chosen before it is decided which distances to reveal. In particular, it is sufficient that every realisation typically becomes globally rigid after randomly permuting the points. This problem was solved precisely by Gir\~ao, Illingworth, Michel, Powierski, and Scott \cite{GIMPS}, and subsequently Montgomery, Nenadov, and Szab\'o \cite{MNS24} proved the stronger result that in fact every injective realisation is globally rigid at the same threshold.
	
	The corresponding problem in $\mathbb{R}^d$ is not interesting for $d \geq 3$, since placing all but two points on a common hyperplane means that a specific distance needs to be revealed for the realisation to be globally rigid (or even rigid). One can still ask when most of the point set can be reconstructed, and Barnes, Petr, Portier, Randall Shaw, and Sergeev \cite{BPPRS} give bounds on the probability $p$ so that $n-o(n)$ points can be reconstructed.
	
	In this paper, we complete the story by establishing sharp thresholds for rigidity in the plane.  In dimension $2$, the question of when every realisation is rigid (but not necessarily globally rigid) is particularly interesting. For instance, the Peaucellier linkage, Burmester's focal mechanism (see Kempe~\cite{Kempe1877}), or the two constructions of flexible realisations of $K_{3,3}$ by Dixon \cite{Dixon} are classical examples of graphs that are generically rigid in the plane, but  admit paradoxically flexible (nongeneric) realisations.  Although the above realisations are all injective, the modern literature has instead focused on \emph{quasi-injective realisations}, which only require all edge distances to be positive.  One reason why this is the more natural setting when considering rigidity (rather than global rigidity) is that the set of injective realisations is not closed under continuous motions that preserve edge lengths, whereas the set of quasi-injective realisations is.  A remarkable theorem of 
	Grasegger, Legersk\'y, and Schicho~\cite{GLS2019} gives a purely combinatorial description of when a graph $G$ has a quasi-injective realisation in the plane. Their characterisation involves a special type of 2-colouring of the edges of $G$, which we now define.  A \emph{NAC-colouring} of $G$ is a surjective map $E(G) \to \{\red{}, \blue{}\}$ such that no cycle of $G$ contains exactly one red edge or exactly one blue edge. 
	Grasegger, Legersk\'y, and Schicho~\cite{GLS2019} prove that a connected graph has a flexible quasi-injective realisation in $\mathbb{R}^2$ if and only if it has a NAC-colouring.

	The following are our main results.  
	
	\begin{theorem} \label{thm:main1}
		In the binomial random graph $G(n,p)$, $p=\sqrt[3]{\frac{2\log n}{n^2}}$ is a sharp threshold for the property of being connected and having no NAC-colouring.  
	\end{theorem}
	In fact, we prove a stronger exact hitting time result.
	By the above discussion, \Cref{thm:main1} also immediately implies the following.  
	
	\begin{theorem} \label{thm:main2}
		In the binomial random graph $G(n,p)$, $p=\sqrt[3]{\frac{2\log n}{n^2}}$ is a sharp threshold for the property of being connected and having no flexible quasi-injective realisation in $\mathbb{R}^2$.  
	\end{theorem}
	
	A \emph{stable cut} in a graph $G$ is a set of vertices $S$ such that $G-S$  is disconnected and $S$ is a stable set.  It is easy to see that if $G$ has a stable cut, then $G$ has a NAC-colouring.  Our proof of~\Cref{thm:main1} also establishes an exact hitting time result for the property of having no stable cut.  It turns out that the threshold is the same as the threshold for having no NAC-colouring. 
	
	\begin{theorem} \label{thm:main3}
		In the binomial random graph $G(n,p)$, $p=\sqrt[3]{\frac{2\log n}{n^2}}$ is a sharp threshold for the property of having no stable cut.  
	\end{theorem}
	
	Gallet, Grasegger, Legerský, and Schicho~\cite{GGLS2021} prove that having a stable cut is equivalent to having a `flex' on the sphere.  Thus,~\Cref{thm:main3} can also be interpreted as a sharp threshold result for rigidity on the sphere.  
	
	Our main theorems also complement results from~\cite{CGHHLN24} on the number of NAC-colourings and stable cuts in various graph classes. For example, in~\cite{CGHHLN24}, it is shown that $K_{n,n}$ has exactly $2^{2n-1}-2$ NAC-colourings.  Thus, there exist extremely dense graphs with many NAC-colourings, which is in stark contrast with the typical behaviour in $G(n,p)$ given by~\Cref{thm:main1}.

	\section{Preliminaries}
	
	We will consider several models for random graphs on $n$ vertices. All are homogeneous, that is, treat all potential edges equally. Throughout the paper, we say an event happens \textit{asymptotically almost surely} (or ``with high probability'') if its probability tends to $1$ as $n\to\infty$. 
	
	The best-known model is the \textit{binomial random graph} (or Gilbert model) $G(n,p)$. In this model a graph on $n$ vertices is constructed by connecting vertices randomly. Each edge is included in the graph with probability $p$ independently from every other edge. Typically we think of $p$ as a function of $n$.
	
	Although the binomial random graph is sometimes also referred to as the Erd\H{o}s--R\'enyi model, the model originally introduced by Erd\H{o}s and R\'enyi is slightly different: it is the graph $G(n;m)$ chosen uniformly at random from among all $n$-vertex graphs with exactly $m$ edges. Results about $G(n,p)$ typically correspond straightforwardly to those about $G(n;m)$ for $m\approx p\binom n2$, but the former is easier to work with.
	
	We will also consider the \textit{random graph process} $(G_t(n))_{t=0}^{\binom{n}{2}}$, a discrete-time process which begins with $G_0(n)$ being the empty graph on $n$ vertices. At each time step, an edge selected uniformly at random from those not in $E(G_{t-1}(n))$ (independently of previous choices) and added to $G_{t-1}(n)$ to give $G_t(n)$. One may equivalently think of this as a sequence of random graphs with $G_t(n)\sim G(n;t)$, coupled so that $G_{t-1}(n)$ is a subgraph of $G_t(n)$ for each $t>0$.
	
	Finally, we briefly consider the \textit{random regular graph} $G_{\mathrm{reg}}(n,k)$, which is chosen uniformly at random from all $k$-regular graphs on $n$ vertices. Here we typically think of $k$ as being fixed (and, if $k$ is odd, restrict $n$ to even values). It is not clear from this description how to produce such a graph in practice. One standard approach is the \textit{configuration model}: give each vertex $k$ ``edge-ends''; pair the $nk$ edge-ends off uniformly at random, creating an edge corresponding to each pair; and condition on the resulting multigraph being simple. For fixed $k$, this event has probability bounded away from $0$; see \cite{BBconfig}.
	
	Jackson, Servatius and Servatius \cite{JSS2007} proved that a random $k$-regular graph, with $k \geq 4$, is asymptotically almost surely globally rigid in $\mathbb{R}^2$. Despite this strong form of rigidity, we can easily see that such graphs have many NAC-colourings and hence flexible realisations.
	
	\begin{proposition}
		A random $k$-regular $n$-vertex graph $G$, with constant $k\geq 4$, asymptotically almost surely has $2^{\Theta(n)}$ NAC-colourings.
	\end{proposition}
	
	\begin{proof}
		Since the number of triangles in $G$ is distributed $\operatorname{Po}((k-1)^3/6)$ \cite{BBconfig}, almost all vertices in $G$ have stable neighbourhoods with high probability.  Let $X$ be a maximal (under inclusion) set of vertices such that the distance between any pair of distinct vertices  in $X$ is at least $4$.  Since $G$ is $k$-regular, for each $x \in X$, there are at most $k(1+(k-1)+(k-1)^2)+1=k^3-k^2+k+1$ vertices at distance at most $3$ from $x$ (including $x$). 
		By the maximality of $X$ we have $|X| \geq \frac{n}{k^3-k^2+k+1}$.
		Deleting any vertices of $X$ that fail to have stable neighbourhoods, we obtain (with high probability) a set $S$ of size at least $\frac{n}{k^3-k^2+k+2}$,
		such that vertices in $S$ have stable neighbourhoods that are pairwise disjoint and have no edges between them.
		Any colouring where each vertex in $S$ has its incident edges monochromatic, each edge with neither end in $S$ is blue, and not all edges of $G$ are blue, is a NAC-colouring, and there are $2^{|S|}-1=2^{\Omega(n)}$ such colourings.  
		The upper bound is trivial since there are only $2^{kn/2}$ ways to $2$-colour the edges of $G$.
	\end{proof}
	
	We say that a function $f(n)$ is a \textit{threshold} for a given property in the binomial random graph if $G(n,p)$ asymptotically almost surely has the property for $p=p(n)\gg f(n)$ and asymptotically almost surely does not have the property for $p=p(n)\ll f(n)$. We say that a threshold is \textit{sharp} if these statements hold instead for $p(n)>(1+\varepsilon) f(n)$ and $p(n)<(1-\varepsilon) f(n)$, for every $\varepsilon>0$. 
	
	We say that a graph property $\mathcal P$ is \textit{monotone} if adding edges to any graph with  property $\mathcal P$ results in another graph with property $\mathcal P$.
	Every monotone property has a threshold \cite{BT87} (but not necessarily a sharp threshold).
	For a monotone property $\mathcal P$, we also define the \textit{hitting time} $\tau_{\mathcal P}$ in the random graph process $(G_t(n))$ to be the smallest $t$ for which $G_t(n)\in \mathcal P$. Since $\mathcal P$ is monotone, $G_t(n)\in \mathcal P$ if and only if $t\geq \tau_\mathcal P$.
	
	\section{The threshold for stable cuts}
	
	In this section we establish the threshold for a graph to have no stable cut. We show that this coincides with the  threshold for every vertex to be in a triangle, that is for no vertex to have a stable neighbourhood. In fact we prove a stronger ``hitting-time'' result. We denote the property that every vertex is in a triangle by $\mathcal T$, and the property of having no stable cut by $\mathcal S$.
	
	We use the known threshold for $\mathcal T$.
	\begin{lemma}[{\cite[Section 7]{Kriv97}}]\label{triangle-threshold}
		In the binomial random graph model,  $\sqrt[3]{\frac{2\log n}{n^2}}$ is a sharp threshold for $\mathcal T$.
	\end{lemma}
	
	We will prove the following strong hitting-time version alluded to in the introduction.
	\begin{theorem}\label{thm:ERnostablecut}
		In the random graph process $(G_t(n))$, asymptotically almost surely $\tau_{\mathcal S}=\tau_{\mathcal T}$.
	\end{theorem}
	\begin{corollary}
		In the binomial random graph model, $\sqrt[3]{\frac{2\log n}{n^2}}$ is a sharp threshold for $\mathcal S$.    
	\end{corollary}
	
	We deduce \Cref{thm:ERnostablecut} from the following result. Suppose that $G$ has a stable cut $S$, but every vertex is in a triangle. Then $G- S$ has at least two components, all of which  have at least two vertices. We call such a cut $S$ a \emph{firm cut}.  Unfortunately, the property of not having a firm cut is not monotone.  For example, if $G$ is obtained from a path $xyz$ by gluing a triangle onto $y$, then $G$ has no firm cut, but $\{y\}$ is a firm cut in $G+xz$.  Thus, we introduce a slight modification, which is easily seen to be monotone. Let $\mathcal S'$ be the property that $G$ has no stable cut $S$ such that $G- S$ either has at least three components or has exactly two components, each with at least two vertices. If $G$ has a stable cut $S$ that is not of this form, then $G-S$ has two components, one of which is a single vertex, and that vertex cannot be in a triangle. Thus $\mathcal S=\mathcal T\cap \mathcal S'$.
	
	\begin{lemma}\label{lem:nobiggercut}There is some $\eta>0$ with the following property. For $p=(1-\eta)\sqrt[3]{\frac{2\log n}{n^2}}$, asymptotically almost surely $G(n,p)\in \mathcal S'$.
	\end{lemma}
	
	\begin{proof}[Proof of \Cref{thm:ERnostablecut}]
		Since $\mathcal S\subseteq \mathcal T$, we must have $\tau_{\mathcal T}\leq \tau_{\mathcal S}$, and it suffices to show that asymptotically almost surely $G_{\tau_{\mathcal T}}(n)\in \mathcal S$. 
		
		Let $M\sim \operatorname{Bin}\left(\binom n2,p\right)$. Then $G_M(n)\sim G(n;M)\sim G(n,p)$ and so asymptotically almost surely $G_M(n)\in \mathcal S'$ by \Cref{lem:nobiggercut} and $G_M(n)\not\in \mathcal T$ by \Cref{triangle-threshold}. Thus asymptotically almost surely $\tau_{\mathcal S'}<\tau_{\mathcal T}$, and so $G_{\tau_{\mathcal T}}(n)\in \mathcal S'$. Since $\mathcal S=\mathcal{T}\cap \mathcal S'$, the result follows.
	\end{proof}
	
	\begin{proof}[Proof of \Cref{lem:nobiggercut}]Fix $\eta$ with $0<\eta<1$ to be chosen later, and set $p=(1-\eta)\sqrt[3]{\frac{2\log n}{n^2}}$. We bound the probability that $G(n,p)\not\in \mathcal S'$.
		
		Suppose $G(n,p)\not\in \mathcal S'$. That is, $G$ has a stable cut $S$ such that $G- S$ either has at least three components or has exactly two components, each with at least two vertices.  We first deal with the case where $G- S$ has at least two non-singleton components. In this case, there exist disjoint non-empty sets $A$, $B$ and $S$ such that $S$ is stable, there are no edges between $A$ and $B$, and each of $G[A]$ and $G[B]$ has at least one edge. We choose such a triple with $A\cup S$ minimal. Any such triple has $G[A]$ connected (otherwise we could move a component from $A$ to $B$), every vertex in $S$ adjacent to some vertex of $A$ (else we could move that vertex from $S$ to $B$), and $2\leq |A|\leq |B|$. 
		
		Fix integers $k$ and $\ell$ with $2\leq k\leq n-k-\ell$. We bound the probability $p_{k,\ell}$ that such a triple exists with $|A|=k$ and $|S|=\ell$. 
		
		A given $k$-set induces a connected graph with probability at most $k^{k-2}p^{k-1}$, since this (by Cayley's formula) is the expected number of spanning trees it contains. A given $\ell$-set has every vertex adjacent to a given $k$-set with probability $(1-(1-p)^k)^\ell<(kp)^\ell$. Thus we obtain
		\begin{align}p_{k,\ell}&\leq \binom{n}{k}\binom{n-k}{l}k^{k-2}p^{k-1}(kp)^\ell(1-p)^{k(n-k-\ell)+\binom{\ell}{2}}\nonumber\\
			&\leq\frac{1}{pk^2}\binom nk\binom n\ell (pk)^{k+\ell}\exp\left(-p\left(k(n-k-\ell)+\binom{\ell}{2}\right)\right)\label{pkl-bound}\\
			&\leq\frac{1}{pk^2}(nep)^k\left(\frac{nepk}{\ell}\right)^\ell\exp\left(-p\left(k(n-k-\ell)+\binom{\ell}{2}\right)\right).\nonumber\end{align}
		Differentiating with respect to $\ell$, we obtain that (for fixed $k$) $\left(\frac{nepk}{\ell}\right)^\ell$ is maximised when $\ell=npk$, so is at most $\exp(npk)$. Furthermore, if $\ell<0.95npk$ then $\left(\frac{nepk}{\ell}\right)^\ell<\exp(0.999npk)$ by direct calculation. Note also that $k\ell-\binom{\ell}{2}$ is maximised (for fixed $k$) when $\ell=k$, so $k\ell-\binom{\ell}{2}< k^2$.
		
		Consequently, whenever $\ell<0.95npk$, we have
		\begin{align*}p_{k,\ell}&\leq \frac{1}{pk^2}(nep)^k\exp(-p(0.001nk-2k^2))\\
			&\leq\frac{1}{pk^2}\exp(-0.001npk+O(k\log n)),\end{align*}
		which is $o(n^{-2})$ since $\log n=o(np)$.
		
		Thus we may assume $\ell\geq 0.95npk$. Now we have
		\begin{align*}p_{k,\ell}&\leq\frac{1}{pk^2}(nep)^k\exp\left(p\left(k^2+k\ell -\binom{\ell}{2}\right)\right)\\
			&\leq n^{(k+2)/3}\exp(-(0.45125-o(1))(k^2n^2p^3)).\end{align*}
		Since $n^2p^3=2(1-\eta)^3\log n$ and $k\geq 2$, this is $o(n^{-2})$ provided $\eta$ is sufficiently small. Thus, taking a union bound over all possible choices of $k$ and $\ell$, with high probability no such triple exists.
		
		Finally, we deal with the remaining case, where $G$ contains a cut $S$ such that $G- S$ has at least three components, with at most one component having more than one vertex. Choosing $A$ to be the union of two singleton components, and taking $S$ to be minimal, we obtain disjoint non-empty sets $A$, $B$ and $S$ such that $S$ is stable, there are no edges between $A$ and $B$, every vertex of $S$ is adjacent to some vertex of $A$, and $|A|=2$. Write $p_\ell$ for the probability that such a triple exists with $|S|=\ell$. Then, arguing as above, we obtain
		\[p_\ell\leq \binom n2\binom n\ell (2p)^\ell\exp\left(-p\left(2(n-2-\ell)+\binom{\ell}{2}\right)\right).\]
		Since this is $p^{-1}$ times our bound on $p_{2,\ell}$ from \eqref{pkl-bound} (the missing factor of $p$ comes from the edge inside $A$ being absent), it is $o(n^{-4/3})$, and again a union bound over possible values of $\ell$ gives the required result.
	\end{proof}
	
	\section{The threshold for NAC-colourings}
	
	In this section we consider when the random graph $G(n,p)$ has a NAC-colouring. Note that having no NAC-colouring is not a monotone property of graphs. For example, the empty graph on $n$ vertices has no NAC-colouring and nor does the complete graph on $n$ vertices, but every tree on $n$ vertices has a NAC-colouring (provided $n \geq 3$).  
	
	We therefore work with the property of being connected with no NAC-colouring, which we denote $\mathcal N$. We claim that $\mathcal N$ is a monotone property. Indeed, suppose $G'$ is a spanning subgraph of $G$, and $G$ does not have property $\mathcal{N}$. If $G$ is disconnected then so is $G'$. If $G$ has a NAC-colouring $c$, then no spanning tree of $G$ can be monochromatic in $c$, since if there is a blue spanning tree then any red edge would create an almost-monochromatic cycle. Therefore either $G'$ is disconnected or $c$ restricted to $G'$ uses both colours and so is a NAC-colouring.
	
	Thus property $\mathcal N$ has a threshold. In this section we prove that this coincides with the thresholds for $\mathcal T$ and $\mathcal S$, again proving the strong hitting-time version.
	\begin{theorem}\label{thm:nacthreshold}
		In the random graph process $(G_t(n))$, asymptotically almost surely $\tau_{\mathcal N}=\tau_{\mathcal T}$.
	\end{theorem}
	\begin{corollary}
		In the binomial random graph model, $\sqrt[3]{\frac{2\log n}{n^2}}$ is a sharp threshold for $\mathcal N$.    
	\end{corollary}
	Recall that with high probability $G_t(n)$ is connected for $t=(1+o(1))\frac{\log n}{n}\ll \tau_{\mathcal T}$ \cite{ER59}, so we will only need to check whether there is a NAC-colouring.
	
	We say that a NAC-colouring $c$ of a graph $G$ is \textit{stable} if
	there is some stable set $S_c$ such that all edges meeting $S_c$ are one colour and all edges not meeting $S_c$ are the other colour.
	
	In some cases, a NAC-colouring can be stable in both ways: there are stable sets $S_c^{\text{red}}$ and $S_c^{\text{blue}}$ such that an edge is red if and only if it meets $S_c^{\text{red}}$ and blue if and only if it meets $S_c^{\text{blue}}$. In this case, $G$ is bipartite with one part being $S_c^{\text{red}}\cup S_c^{\text{blue}}$. Bipartite graphs have many other stable NAC-colourings, as shown by the following observation.
	\begin{lemma}\label{lem:bipartite}Let $G$ be a bipartite graph, and let $S$ be any stable set that meets at least one edge but is not a vertex cover. Then the colouring defined by setting edges red if they meet $S$, and blue otherwise, is a stable NAC-colouring.\end{lemma}
	\begin{proof}At least one edge meets $S$ and at least one does not, so both colours are used. Consider any cycle of $G$. Since $S$ is stable, the cycle contains twice as many red edges as it contains vertices of $S$, and in particular has an even number of red edges. Since $G$ is bipartite, the cycle must also have an even number of blue edges. Therefore there are no almost-monochromatic cycles, and this is a NAC-colouring; it is stable by definition. 
	\end{proof}

	Consider the property $\mathcal N'$ of being connected with no non-stable NAC-colouring. Unfortunately this property is not monotone. For example, all NAC-colourings of the graph with vertex set $\{v,w,x,y,z\}$ and edge set $\{vx,wx,xy,xz,yz\}$ are stable, but if the edge $vw$ is added there is only one NAC-colouring (up to swapping colours), which is not stable. However, we can show that $\mathcal N'$ is ``almost monotone'' in the sense that only adding edges that fit a certain pattern can violate monotonicity, as we make precise in the following lemma.
	
	\begin{lemma}\label{almost-monotone}Suppose $G$ is a spanning subgraph of $G'$, and $G\in\mathcal N'$ but $G'\not\in\mathcal N'$. Then there is a stable NAC-colouring $c$ of $G$ such that $|S_c|=2$ and $G'$ contains the edge between the vertices of $S_c$.\end{lemma}
	\begin{proof}Suppose not, and take a counterexample where $|E(G')\setminus E(G)|$ is as small as possible. Suppose $|E(G')\setminus E(G)|\geq 2$, and let $G''$ be any graph with $E(G)\subset E(G'')\subset E(G')$. If $G''\not\in\mathcal N'$ then by minimality there is a stable NAC-colouring of $G$ with the required properties with respect to $G''$, but then since $E(G'')\subset E(G')$ it also has the required properties with respect to $G'$. Thus $G''\in \mathcal N'$ and there exists a stable NAC-colouring $c''$ of $G''$ with the required properties with respect to $G$. Let $c$ be the restriction of $c''$ to $E(G)$. Clearly $c$ still has no almost-monochromatic cycles. Since $c''$ is a NAC-colouring, it has no spanning monochromatic component. Therefore, since $G$ is connected, $c$ uses both colours. Without loss of generality, the red edges of $c''$ are precisely those meeting $x$ or $y$ for some $x,y$ with $xy\in E(G')\setminus E(G'')$. Thus also $xy\in E(G')\setminus E(G)$, and $c$ has the required properties.
		
		Thus we may assume $G=G'\setminus xy$ for some edge $xy$. Since $G$ is connected, $G'$ is connected and $xy$ is not a bridge. Since $G'\not\in \mathcal N'$, $G'$ has a non-stable NAC-colouring $c'$. Since $xy$ was not a bridge, $xy$ cannot be the only edge of its colour, and hence $c'$ induces a NAC-colouring $c$ of $G$. Since $G\in\mathcal N'$, $c$ is stable, and without loss of generality the red edges are precisely those that meet some set $S$ which is stable in $G$.
		
		Suppose $\{x,y\}\subseteq S$. Then we may recolour the edges meeting $S\setminus\{x,y\}$ blue to obtain a colouring with the required properties. Thus we may assume without loss of generality $x\not\in S$. Since $c'$ is not stable, either $y\in S$ and $c'$ colours $xy$ blue, or $y\not\in S$ and $c'$ colours $xy$ red. In either case, there exists an $x$-$y$ path in $G$ since $xy$ was not a bridge of $G'$, and any such path must meet $S\setminus\{y\}$ in order to avoid creating a cycle in $G'$ with exactly one red edge under $c'$. Thus $G$ is the union of two graphs $G_1,G_2$, which intersect on $S'\subseteq S\setminus\{y\}$, and which contain $x$ and $y$ respectively.
		
		Consider the colouring $c_1$ of $G$ in which an edge is red if it meets $S'$ and is in $G_1$. We first observe that this is a NAC-colouring of $G$. It clearly uses both colours. Consider any cycle in $G$. If it lies entirely in $G_2$, all its edges are blue. If it uses edges from both $G_1$ and $G_2$, it uses at least two red edges from $G_1$ and at least two blue edges from $G_2$. If it lies entirely in $G_1$, it uses an even number of red edges, so it remains to show it cannot have exactly one blue edge. However, it uses at least as many blue edges in $c_1$ as in $c$, and the difference between the two is twice the number of times the cycle reaches $S\setminus S'$, so even. Since the cycle is not almost-monochromatic under $c$, the claim follows. Now since $G\in\mathcal N'$, we know $c_1$ is stable. Thus there is a set $T$ meeting all edges of one colour, and no others. Clearly $T$ is disjoint from $S'$. If all edges of $G_1$ are red, then $G_1$ is bipartite, with $S'$ being in one part. Otherwise $T$ must meet all blue edges, and $c_1$ restricted to $G_1$ is a NAC-colouring. Since an edge of $G_1$ is red if and only if it meets $S'$, and blue if and only if it meets $T\cap V(G_1)$, we again conclude that $G_1$ is bipartite with $S'$ being in one part. Similarly $G_2$ is bipartite with $S'$ being in one part, from which it follows that $G$ is bipartite. Now Lemma \ref{lem:bipartite} with $S=\{x,y\}$ gives a NAC-colouring with the desired properties.\end{proof}
	
	We now argue that it is sufficient to show the following.
	
	\begin{lemma}\label{noothernac}
		There is some $\eta$ with $0<\eta<1/2$ having the following property. Fix $p=(1-\eta)\sqrt[3]{\frac{2\log n}{n^2}}$, and let $G_1,G_2$ be independent copies of $G(n,p)$. Then,  asymptotically almost surely, $G_1$ is connected and for every NAC-colouring $c$ of $G_1$:
		\begin{itemize}
			\item $c$ is stable; and
			\item if $|S_c|=2$ then the edge between the vertices of $S_c$ is not in $G_2$.
		\end{itemize}
	\end{lemma}
	
	\begin{proof}[Proof of Theorem \ref{thm:nacthreshold}]
		Clearly $\mathcal N\subseteq \mathcal N'$. If $G\in \mathcal N'\setminus\mathcal N$ then $G$ has a stable NAC-colouring $c$, and consequently any vertex of $S_c$ is in no triangle. Hence $\mathcal N=\mathcal N'\cap \mathcal T$, and so $\tau_{\mathcal N}\geq \tau_{\mathcal T}$.
		
		Now let $p$, $G_1$ and $G_2$ be as defined in Lemma \ref{noothernac}, and let $T_1=e(G_1)$ and $T_2=e(G_1\cup G_2)$. We have $T_1\sim\operatorname{Bin}\bigl(\binom{n}{2},p\bigr)$ and $T_2\sim\operatorname{Bin}\bigl(\binom{n}{2},2p-p^2\bigr)$, so Lemma \ref{triangle-threshold} implies that asymptotically almost surely $T_1<\tau_{\mathcal T}<T_2$; assume henceforth that this holds. 
		
		Since $(G_{T_1}(n),G_{T_2}(n))\sim(G_1,G_1\cup G_2)$, asymptotically almost surely $G_{T_1}(n)\in\mathcal N'$ by Lemma \ref{noothernac}. By Lemma \ref{almost-monotone}, if $G_{\tau_{\mathcal T}}(n)\not\in \mathcal N'$ then there exists a stable NAC-colouring $c$ of $G_{T_1}(n)$ with $|S_c|=2$ and an edge between the vertices of $S_c$ in $G_{\tau_{\mathcal T}}(n)$, and hence in $G_{T_2}(n)$. But by Lemma \ref{noothernac}, asymptotically almost surely no such $c$ exists.
		
		Thus $G_{\tau_{\mathcal T}}(n)\in\mathcal N'\cap \mathcal T=\mathcal N$, giving $\tau_{\mathcal N}\leq \tau_{\mathcal T}$, asymptotically almost surely.
	\end{proof}

	We now work towards the proof of Lemma \ref{noothernac}. We first mention some simple properties of $G(n,p)$. For $p$ as in Lemma \ref{noothernac}, the largest stable set of $G$ asymptotically almost surely has $(2-o(1))p^{-1}\log np<n^{2/3}\log^{2/3}n$ vertices (e.g.\ see a more precise result in \cite{Frieze}). The upper bound (which is all we shall require) is straightforward, and we can similarly show that any two sets of this size have an edge between.
	
	\begin{lemma}For any $p=p(n)$ satisfying $np\to\infty$, with high probability there are no two disjoint sets of size at least $2p^{-1}\log np$ having no edges between them in $G(n,p)$.
	\end{lemma}
	\begin{proof}
		Write $k=\lceil 2p^{-1}\log np\rceil$. Since there are $\binom{n}{k}\binom{n-k}{k}\leq \binom{n}{k}^2$ choices of the two sets, each with $k^2$ potential edges between, the probability of two such $k$-sets existing is at most 
		\[\binom{n}{k}^2(1-p)^{k^2}\leq \left(\frac{ne}{k}\exp(-pk/2)\right)^{2k}\leq (e/\log np)^{2k},\]
		again using $(ne/k)^k \geq \binom{n}{k}$.
		Since $np\to\infty$, this is $o(1)$, as required.
	\end{proof}
	
	Thus we have the following.
	\begin{corollary}\label{p1p2}
		For $p$ as in Lemma \ref{noothernac}, setting $z=n^{2/3}\log^{2/3}n$, asymptotically almost surely in $G(n,p)$:
		\begin{enumerate}[label=P\arabic*]
			\item\label{p:indep} there is no stable set of size $z$, and
			\item\label{p:bipart} there are no two disjoint sets of size $z$ having no edges between them. 
		\end{enumerate}
	\end{corollary}
	
	We next give a structural result on NAC-colourings of graphs satisfying these properties. For a NAC-colouring, we define a \textit{monochromatic component} to be a maximal monochromatic connected subgraph, and a \textit{red component} to be a monochromatic component all of whose edges are red (and likewise for blue). Note that if two monochromatic components intersect, they must be of opposite colours, and consequently their intersection is a stable set, since otherwise an edge in the intersection together with a path in the opposite-coloured component contradicts the colouring being a NAC-colouring.
	\begin{lemma}\label{onebigcomp}Let $G$ be any graph satisfying properties \ref{p:indep} and \ref{p:bipart} for some $z$. Then in any NAC-colouring of $G$, some monochromatic component contains at least $n-8z$ vertices.
	\end{lemma}
	\begin{proof}Suppose there is a red component $R$ covering at least $2z$ vertices and a blue component $B$ covering at least $2z$ vertices. Then $V(R)\cap V(B)$ is a stable set, so $|V(R)\cap V(B)|<z$. There are no edges between $V(R)\setminus V(B)$ and $V(B)\setminus V(R)$, since there is no red edge leaving $V(R)$ and no blue edge leaving $V(B)$. Thus 
		\[\min\{|V(R)\setminus V(B)|,|V(B)\setminus V(R)|\}=\min\{|V(R)|,|V(B)|\}-|V(R)\cap V(B)|>z,\] 
		a contradiction of \ref{p:bipart}.
		Thus, without loss of generality, we may assume no red component has size at least $2z$. 
		
		For convenience, if $\mathcal C$ is a collection of monochromatic components, we use $V(\mathcal C)$ to denote $\bigcup_{C\in \mathcal C}V(C)$.
		Suppose we can partition the blue components into sets $\mathcal{B}_1,\mathcal{B}_2$, such that $|V(\mathcal B_1)|\geq|V(\mathcal B_2)|\geq 4z$. Then we claim that we can partition the red components into $\mathcal{R}_1,\mathcal{R}_2$ such that $|V(\mathcal B_i)\setminus V(\mathcal R_i)|\geq z$ for each $i$. This will give a contradiction, since there is no edge between the two sets in either colour. To justify the claim, order the red components by $|V(R)\cap V(\mathcal B_1)|/|V(R)\cap V(\mathcal B_2)|$. Put the first $a$ red components in this ordering (that is, those with smallest ratios) into $\mathcal R_1$ and the remainder into $\mathcal R_2$. This ensures that 
		\begin{equation}\label{eqn:ratios}
			\frac{|V(\mathcal R_1)\cap V(\mathcal B_1)|}{|V(\mathcal R_1)\cap V(\mathcal B_2)|}\leq\frac{|V(\mathcal R_2)\cap V(\mathcal B_1)|}{|V(\mathcal R_2)\cap V(\mathcal B_2)|},
		\end{equation}
		from which we obtain
		\begin{align*}
			\frac{|V(\mathcal R_1)\cap V(\mathcal B_1)|}{|V(\mathcal B_1)|}&\leq\frac{|V(\mathcal R_1)\cap V(\mathcal B_1)|}{|V(\mathcal R_1\cup\mathcal R_2)\cap V(\mathcal B_1)|}\\
			&\leq\frac{|V(\mathcal R_1)\cap V(\mathcal B_2)|}{|V(\mathcal R_1\cup\mathcal R_2)\cap V(\mathcal B_2)|}\\
			&\leq 1-\frac{|V(\mathcal R_2)\cap V(\mathcal B_2)|}{|V(\mathcal B_2)|},
		\end{align*}
		where the first and third inequalities follow from the fact $V(\mathcal R_1)\cap V(\mathcal B_i)$ and $V(\mathcal R_2)\cap V(\mathcal B_i)$ are disjoint for $i\in[2]$, and the second inequality follows from \eqref{eqn:ratios}.
		
		Choose $a$ maximal such that $|V(\mathcal R_1)\cap V(\mathcal B_1)|\leq |V(\mathcal B_1)|/2+z$. Since incrementing $a$ changes $|V(\mathcal R_1)\cap V(\mathcal B_1)|$ by at most the order of some red component, which is at most $2z$, we must have $|V(\mathcal R_1)\cap V(\mathcal B_1)|\geq|V(\mathcal B_1)|/2-z$, and thus 
		\[|V(\mathcal R_2)\cap V(\mathcal B_2)|\leq \frac{|V(\mathcal B_2)|}{2}+z\frac{|V(\mathcal B_2)|}{|V(\mathcal B_1)|}\leq\frac{|V(\mathcal B_2)|}{2}+z.\]
		Thus $|V(\mathcal B_i)\setminus V(\mathcal R_i)|\geq 2z-z$ for each $i$, proving the claim.
		
		It follows that there is no way to partition the blue component so that each part covers at least $4z$ vertices, and hence that at most $4z$ vertices are in blue components other than the largest one.
		
		To complete the proof, we argue that at most $4z$ vertices are not in any blue component. Suppose this is not the case, and let $X$ be these vertices. Choose a minimal collection of red components that cover at least $z$ vertices in $X$, and let $Y$ be the set of vertices in these components. Since each red component covers at most $2z$ vertices, $|Y|\leq 3z$. But then $Y$ and $X\setminus Y$ are sets of size at least $z$ with no edge between, contradicting \ref{p:bipart}.
	\end{proof}
	
	We are now ready to prove the main result.
	\begin{proof}[Proof of Lemma \ref{noothernac}]
		Recall that $G_1$ is asymptotically almost surely connected. Suppose that $G_1$ has at least one NAC-colouring $c$ which is either non-stable or is stable with $|S_c|=2$ and an edge between the vertices of $S_c$ in $G_2$.  Choose such a colouring minimising the total number of monochromatic components. 
		
		By \Cref{lem:nobiggercut}, we may assume that every stable cut of $G_1$ separates it into two components, one of which is a singleton.
		
		By \Cref{onebigcomp} and \Cref{p1p2}, without loss of generality, some blue component $B$ covers at least $n-8z$ vertices of $G_1$, where $z=n^{2/3}\log^{2/3}n$. Suppose $G_1-V(B)$ is disconnected. Then choose some (connected) component $C$ of $G_1-V(B)$, and recolour all edges not meeting $C$ blue. Any cycle including a recoloured edge is either entirely blue or reaches $C$. In the latter case, since all edges leaving $C$ go to $B$, and are therefore red, the cycle has at least two red edges; since it also meets some vertex outside $V(C)\cup V(B)$, it has at least two blue edges. Thus this modified colouring is still a NAC-colouring. However it has fewer monochromatic components, since all monochromatic components meeting or containing recoloured edges have been merged with $B$. By our original choice of NAC-colouring, the modified colouring must be stable for every choice of $C$. It follows that $G_1-V(B)$ is either connected or consists of exactly two non-adjacent vertices, and in the latter case there is an edge between these vertices in $G_2$.
		
		Number the red components meeting $B$ as $R_1,\ldots,R_\ell$, and for each $i$ write $S_i=V(R_i)\cap V(B)$ and $T_i=V(R_i)\setminus V(B)$. Then each $S_i$ is a stable set which separates $T_i$ from $V(B)\setminus S_i$ (and we may therefore assume $|S_i|\leq z$ by \ref{p:indep}). Furthermore, every vertex in $S_i$ is adjacent to some vertex of $T_i$, since it meets an edge of $R_i$. We may assume that the sets $T_i$ partition $V(G)\setminus V(B)$, since any leftover vertices may be added arbitrarily to any of the $T_i$. With this assumption, each vertex in $T_i$ has no neighbours outside $S_i\cup\bigcup_{j}T_j$. Note that $\ell\geq 2$, since otherwise the colouring being non-stable means $|T_1|>1$ and hence $S_1$ is a stable cut separating $G$ into two non-singleton components, contradicting \Cref{lem:nobiggercut}. Furthermore, either $G_1[\bigcup_jT_j]$ is connected or $\ell=2$, $|T_1|=|T_2|=1$ and $G_2[T_1\cup T_2]$ is connected.
		
		We now bound the probability of a system of sets with these properties existing. Write $t_i=|T_i|$ and $t=\sum_it_i$, and let $\ell$ and $t_1\geq \cdots\geq t_\ell$ be fixed (reordering if necessary). For convenience we write $\mathbf{t}=(t_1,\ldots,t_\ell)$. Let $k$ be the number of $t_i$ which equal $1$ (so that $t_i\geq 2$ if and only if $i\leq \ell-k$). Fix an ordering of the vertices, and assume the $k$ singleton sets are numbered consistently with this ordering. 
		
		First consider the number of ways to choose the sets $(T_i,S_i)_{i\leq \ell}$. This is at most
		\[\left(\prod_{i=1}^{\ell-k}\binom{n}{t_i}\sum_{s=1}^z\binom{n}{s}\right)\times\binom{n}{k}\left(\sum_{s=1}^z\binom{n}{s}\right)^k,\]
		where the variable $s$ corresponds to the size of the set $S_i$ being chosen. Note that the probability of every vertex in $S_i$ being adjacent to some vertex in $T_i$ is 
		\[\left(1-(1-p)^{t_i}\right)^{|S_i|}\leq (pt_i)^{|S_i|},\] 
		whereas the probability of $T_i$ being non-adjacent to $V(B)\setminus S_i$ is 
		\[(1-p)^{t_i(n-t-|S_i|)}\leq \exp(-pt_i(n-t-|S_i|)).\]
		Likewise the probability of $S_i$ being stable is at most $\exp\bigl(-p\binom {|S_i|}2\bigr)$.
		
		Taking these facts into account, the expected number of choices for which these properties are satisfied, $Q(\mathbf{t})$, is at most
		\[\left(\prod_{i=1}^{\ell}\sum_{s=1}^z\binom{n}{s}(pt_i)^se^{-pt_i(n-s-t)-p\binom{s}{2}}\right)\left(\prod_{i=1}^{\ell-k}\binom n{t_i}\right)\binom{n}{k}.
		\]
		\begin{claim}\label{Qclaim}$Q(\mathbf{t})\leq n^{-(1-10\varepsilon)t}\binom nk$, provided $n$ is sufficiently large (in terms of $\varepsilon$).\end{claim}
		
		Assuming \Cref{Qclaim}, we now argue that the expected number of choices for which the above properties are satisfied, and additionally either $G_1[\bigcup_j T_j]$ is connected or $\mathbf{t}=(1,1)$ and there is an edge between $T_1$ and $T_2$ in $G_1\cup G_2$, is small. We first deal with the latter case: for $\mathbf{t}=(1,1)$ the expected number of such choices is at most $2pQ((1,1))\leq n^{20\varepsilon}p=o(n^{1/2})$ for $\varepsilon$ sufficiently small. 
		
		If $t\geq 3k$ then the bound from the claim immediately gives $n^{-(2/3-10\varepsilon) t}$. If not, we factor in the probability of $G_1[\bigcup_j T_j]$ being connected (note that this is independent of the events previously considered). This is at most $t^{t-2}p^{t-1}$, the expected number of spanning trees on $t$ vertices, and trivially at most $1$. By considering the cases $tp\geq 1$ and $tp<1$ separately, we have $\min\{1,t^{t-2}p^{t-1}\}\leq t^kp^{k-1}\leq (3k)^kp^{k-1}$. Now we have 
		\[(3k)^kp^{k-1}Q(\mathbf t)\leq p^{-1}(3enp)^kn^{-(1-10\varepsilon)t}.\]
		Noting that $3enp\leq n^{1/3+\varepsilon}$ for $n$ sufficiently large, $p^{-1}\leq n^{2/3}$ and $k\leq t$, this is at most $n^{-(2/3-11\varepsilon)(t-1)}$. We now need to sum this bound over all possible choices of $\mathbf{t}$, obtaining $\sum_{t\geq 2}p(t)n^{-(2/3-11\varepsilon)(t-1)}$, where $p(t)$ is the partition function, which is known to be subexponential. Thus this overall bound is $o(n^{-1/2})$.
		
		To complete the proof, it only remains to justify \Cref{Qclaim}.
		\begin{proof}[Proof of \Cref{Qclaim}]
			Recall that $t=n-|V(B)|\leq 8n^{2/3}\log^{2/3}n$. We first give an argument to cover the case that $t$ is not too close to this, then give the additional steps needed to deal with larger $t$. We use the fact that, for any $c>0$, the function $c^xx^{-x}$ is increasing for $x<c/e$, and decreasing thereafter. 
			We bound 
			\[f(t_i,t):=\sum_{s=1}^z\binom{n}{s}(pt_i)^se^{-pt_i(n-s-t)-p\binom{s}{2}},\] which we split into two sums depending on the value of $s$. Let 
			\[f_1(t_i,t)=\sum_{\substack{(1-\varepsilon)pt_in\leq s\leq\\\min\{z,(1+\varepsilon)pt_in\}}}\binom{n}{s}(pt_i)^se^{-pt_i(n-s-t)-p\binom{s}{2}},\]
			and $f_2(t_i,t)=f(t_i,t)-f_1(t_i,t)$. 
			For $(1-\varepsilon)pt_in\leq s$, we have
			\[e^{-p\binom s2}\leq e^{-n^2p^3t_i^2(1/2-\varepsilon)}\leq n^{-(1-3\varepsilon)t_i^2},\]
			whereas
			\[\sum_{\substack{(1-\varepsilon)pt_in\leq s\leq\\\min\{z,(1+\varepsilon)pt_in\}}}\binom ns(pt_i)^s\leq(1+pt_i)^n\leq e^{pt_in}.\]
			Thus
			\begin{equation}f_1(t_i,t)\leq e^{pt_i(\min\{z,(1+\varepsilon)pt_in\}+t)}n^{-(1-3\varepsilon)t_i^2}.\label{f1}\end{equation}
			
			Suppose $t\leq n^{2/3}$ and $t_i\leq \bfrac{n}{\log n}^{1/3}$. Then $(1+\varepsilon) npt_i\leq 2n^{2/3}$, and so \[p(\min\{z,(1+\varepsilon)pt_in\}+t)\leq 3\log^{1/3}n<\varepsilon\log n\] for $n$ sufficiently large. Thus we obtain $f_1(t_i,t)\leq n^{-(1-3\varepsilon)t_i^2+\varepsilon t_i}$. This is $n^{-(1-4\varepsilon)}$ if $t_i=1$ and is otherwise at most $n^{-(2-7\varepsilon)t_i}$. 
			
			Now suppose $t_i\geq 21$. Since $t\leq 8z$, we have $p(\min\{z,(1+\varepsilon)pt_in\}+t)\leq 9pz\leq 18\log n$. Thus we obtain $f_1(t_i,t)\leq n^{-(1-3\varepsilon)t_i^2+18t_i}\geq n^{-2t_i}$.
			
			Next we turn to $f_2(t_i,t)$. If $s<(1-\varepsilon)pt_in$ then
			\[\binom ns(pt_i)^se^{-pt_in}\leq\bfrac{ept_in}{s}^se^{-pt_in}\leq \bfrac{e}{1-\varepsilon}^{(1-\varepsilon)pt_in}e^{-pt_in}\leq e^{-2\delta pt_in}\]
			for some $\delta>0$ depending on $\varepsilon$, and similarly for $s>(1+\varepsilon)pt_in$. Since $s+t\leq 9z\ll \delta n$ and $pn\gg \log n$, 
			\[\binom ns(pt_i)^se^{-pt_i(n-s-t)}\leq e^{-\delta pt_in}\leq n^{-2t_i-1}.\]
			Summing this over values of $s\leq z$ that are not counted in $f_1$, we obtain $f_2(t_i,t)\leq n^{-2t_i}$.
			
			For $t\leq n^{2/3}$, this proves the claim, since $f(t_i,t)=f_1(t_i,t)+f_2(t_i,t)$ is at most $n^{-(1-5\varepsilon)}$ if $t_i=1$ and is otherwise at most $n^{-(2-8\varepsilon)t_i}$, so 
			\[Q(\mathbf{t})\leq \left(\prod_{i:t_i>1}n^{t_i}f(t_i,t)\right)\left(\prod_{i:t_i=1}f(1,t)\right)\binom nk\leq n^{-(1-8\varepsilon)t}\binom nk.\]
			
			If $t>n^{2/3}$, we need to bound $Q(\mathbf{t})$ more carefully. We have
			\[Q(\mathbf{t})\leq\sum_{s_1,\ldots,s_\ell}\left(\prod_{i=1}^{\ell}\binom{n}{s_i}(pt_i)^{s_i}e^{-pt_i(n-s_i-t)-p\binom{s_i}{2}}\right)\left(\prod_{i=1}^{\ell-k}\binom n{t_i}\right)\binom{n}{k},
			\]
			where the sum is taken over all $\ell$-tuples $s_1,\ldots,s_\ell$ satisfying $1\leq s_i\leq z$ for each $i$ and $\sum_is_i\leq n$. We break this sum up by considering, for each $i$, whether or not $(1-\varepsilon)pt_in\leq s_i\leq(1+\varepsilon)pt_in$. Fix a sequence $\sigma_1,\ldots,\sigma_\ell$ with $\sigma_i=1$ if this inequality is satisfied, and $\sigma_i=2$ otherwise.
			
			Note that 
			$\sum_{i:\sigma_i=1}(1-\varepsilon)pt_in\leq n$ and thus \begin{equation}\sum_{i:\sigma_i=1}t_i\leq \frac{t}{\log^{1/3}n}.\label{few}\end{equation}
			We thus obtain
			\[Q(\mathbf{t})\leq\sum_{\sigma_1,\ldots,\sigma_\ell}\left(\prod_{i=1}^\ell f_{\sigma_i}(t_i,t)\right)\left(\prod_{i=1}^{\ell-k}\binom n{t_i}\right)\binom{n}{k},\]
			where the sum is taken over sequences satisfying \eqref{few}. As before, we have $f_2(t_i,t)\leq n^{-2t_i}$. To bound $f_1(t_i,t)$, we use \eqref{f1} and the fact that $t\leq 8z$ to obtain $f_1(t_i,t)\leq n^{18t_i}$. Thus
			\[Q(\mathbf{t})\leq\sum_{\sigma_1,\ldots,\sigma_\ell}n^{-t+20\sum_{i:\sigma_i=2}t_i}\leq 2^\ell n^{-(1-\varepsilon)t}\leq n^{-(1-2\varepsilon)t},\]
			using \eqref{few} and $\ell\leq t$.
		\end{proof}
		This completes the proof of~\Cref{noothernac}.
	\end{proof}
	
	\section{Concluding remarks}
	
	Our threshold results give a complete picture of the emergence of stable cuts and NAC-colourings in random graphs. However, there are several avenues for further work in this area.
	
	Firstly, what if we relax stable cuts to allow other types of sparse cuts? In particular, what is the correct analogue of Theorem \ref{thm:ERnostablecut} for cuts that induce a forest? Forest cuts were introduced by Chernyshev, Rauch, and Rautenbach \cite{chernyshev.etal_2024}, and further studied in \cite{botler.etal_2024,BNSVRV25}. A plausible conjecture is that the threshold for no forest cut coincides with the threshold for every neighbourhood to contain a cycle. Noting that a graph is $1$-degenerate if and only if it is a forest, one could more generally consider $k$-degenerate cuts. Alternative relaxations of the stable cut problem include disconnected cuts (see \cite{disco-cut}).
	
	Secondly, while Theorem \ref{thm:nacthreshold} tells us when the random graph has no flexible quasi-injective realisation, it would be interesting to also determine the threshold for having no flexible \emph{injective} realisation. An immediate obstacle is the lack of a known combinatorial classification of which graphs have this property.  That is, there is no analogue of NAC-colourings for flexible injective realisations. However, Grasegger, Legersk\'y, and Schicho \cite{GLSinjective} do have some partial results in this direction.
	
	Thirdly, we can move from an entirely random setting to a partially random one. 
	The study of properties of randomly perturbed graphs was introduced by Bohman, Frieze, and Martin \cite{rand-pet-ham} in order to build a bridge between classical results on Hamiltonian cycles in extremal versus probabilistic settings. Randomly perturbed graphs have since been studied in a wide range of other contexts. A particularly relevant example to our work is the existence of triangle factors \cite{rand-pet-tri}. The randomly perturbed model starts from a deterministic graph with some degree condition, and sprinkles on a small number of additional edges at random. Typically, the desired property can be obtained even when the degree bound and number of random edges are both much lower than would be sufficient individually. In the worst case, starting from a graph taken from some suitable class which allows flexible realisations, how many random edges should be added to eliminate all NAC-colourings? 
	
	\section*{Acknowledgements}
	
	This project originated from the Fields Institute Focus Program on Geometric Constraint Systems in Toronto and a substantial part of the work was done during a Research-in-Groups Programme funded by the International Centre for Mathematical Sciences, Edinburgh.
	The authors are grateful to both institutions for their hospitality and generous financial support, and to Jan Legersk\'y for helpful discussions.
	
	K.\,C.\ was supported by the Australian Government through the Australian Research
	Council’s Discovery Projects funding scheme (project DP210103849). 
	A.\,N.\ was partially supported by EPSRC grant EP/X036723/1. T.\,H.\ is supported by the Institute for Basic Science (IBS-R029-C1).


\begin{thebibliography}{10}
	
	\bibitem{AsimowRoth}
	L.~Asimow and B.~Roth.
	\newblock The rigidity of graphs.
	\newblock {\em Transactions of the American Mathematical Society}, 245:279--289, 1978.
	\newblock \href {https://doi.org/10.2307/1998867} {\path{doi:10.2307/1998867}}.
	
	\bibitem{BPPRS}
	Douglas Barnes, Jan Petr, Julien Portier, Benedict {Randall Shaw}, and Alan Sergeev.
	\newblock Reconstructing almost all of a point set in $\mathbb{R}^d$ from randomly revealed pairwise distances, 2024.
	\newblock arXiv preprint.
	\newblock \href {https://doi.org/10.48550/arXiv.2401.01882} {\path{doi:10.48550/arXiv.2401.01882}}.
	
	\bibitem{BT22}
	Itai Benjamini and Elad Tzalik.
	\newblock Determining a points configuration on the line from a subset of the pairwise distances, 2022.
	\newblock arXiv preprint.
	\newblock \href {https://doi.org/10.48550/arXiv.2208.13855} {\path{doi:10.48550/arXiv.2208.13855}}.
	
	\bibitem{BNSVRV25}
	Ilya~I. Bogdanov, Elizaveta Neustroeva, Georgy Sokolov, Alexei Volostnov, Nikolay Russkin, and Vsevolod Voronov.
	\newblock On forest and bipartite cuts in sparse graphs, 2025.
	\newblock arXiv preprint.
	\newblock \href {https://doi.org/10.48550/arXiv.2505.16179} {\path{doi:10.48550/arXiv.2505.16179}}.
	
	\bibitem{rand-pet-ham}
	Tom Bohman, Alan Frieze, and Ryan Martin.
	\newblock How many random edges make a dense graph {H}amiltonian?
	\newblock {\em Random Structures Algorithms}, 22(1):33--42, 2003.
	\newblock \href {https://doi.org/10.1002/rsa.10070} {\path{doi:10.1002/rsa.10070}}.
	
	\bibitem{BT87}
	B.~Bollob\'as and A.~Thomason.
	\newblock Threshold functions.
	\newblock {\em Combinatorica}, 7(1):35--38, 1987.
	\newblock \href {https://doi.org/10.1007/BF02579198} {\path{doi:10.1007/BF02579198}}.
	
	\bibitem{BBconfig}
	B\'ela Bollob\'as.
	\newblock A probabilistic proof of an asymptotic formula for the number of labelled regular graphs.
	\newblock {\em European Journal of Combinatorics}, 1(4):311--316, 1980.
	\newblock \href {https://doi.org/10.1016/S0195-6698(80)80030-8} {\path{doi:10.1016/S0195-6698(80)80030-8}}.
	
	\bibitem{botler.etal_2024}
	F.~Botler, Y.~S. Couto, C.~G. Fernandes, E.~F. de~Figueiredo, R.~Gómez, V.~F. dos Santos, and C.~M. Sato.
	\newblock Extremal {{Problems}} on {{Forest Cuts}} and {{Acyclic Neighborhoods}} in {{Sparse Graphs}}, 2024.
	\newblock arXiv preprint.
	\newblock \href {https://doi.org/10.48550/arXiv.2411.17885} {\path{doi:10.48550/arXiv.2411.17885}}.
	
	\bibitem{rand-pet-tri}
	Julia B\"ottcher, Olaf Parczyk, Amedeo Sgueglia, and Jozef Skokan.
	\newblock Triangles in randomly perturbed graphs.
	\newblock {\em Combin. Probab. Comput.}, 32(1):91--121, 2023.
	\newblock \href {https://doi.org/10.1017/s0963548322000153} {\path{doi:10.1017/s0963548322000153}}.
	
	\bibitem{chernyshev.etal_2024}
	Vsevolod Chernyshev, Johannes Rauch, and Dieter Rautenbach.
	\newblock Forest cuts in sparse graphs.
	\newblock {\em Discrete Math.}, 348(11):Paper No. 114594, 6, 2025.
	\newblock \href {https://doi.org/10.1016/j.disc.2025.114594} {\path{doi:10.1016/j.disc.2025.114594}}.
	
	\bibitem{CGHHLN24}
	Katie Clinch, Dániel Garamvölgyi, John Haslegrave, Tony Huynh, Jan Legerský, and Anthony Nixon.
	\newblock Stable cuts, {NAC}-colourings and flexible realisations of graphs, 2024.
	\newblock arXiv preprint.
	\newblock \href {https://doi.org/10.48550/arXiv.2412.16018} {\path{doi:10.48550/arXiv.2412.16018}}.
	
	\bibitem{connelly05}
	Robert Connelly.
	\newblock Generic global rigidity.
	\newblock {\em Discrete Comput. Geom.}, 33(4):549--563, 2005.
	\newblock \href {https://doi.org/10.1007/s00454-004-1124-4} {\path{doi:10.1007/s00454-004-1124-4}}.
	
	\bibitem{Dixon}
	A.~C. Dixon.
	\newblock On certain deformable frameworks.
	\newblock {\em The Messenger of Mathematics}, 29(2):1--21, 1899.
	
	\bibitem{ER59}
	P.~Erd\H{o}s and A.~R\'enyi.
	\newblock On random graphs. {I}.
	\newblock {\em Publicationes Mathematicae Debrecen}, 6:290--297, 1959.
	\newblock \href {https://doi.org/10.5486/pmd.1959.6.3-4.12} {\path{doi:10.5486/pmd.1959.6.3-4.12}}.
	
	\bibitem{ER61}
	P.~Erd\H{o}s and A.~R\'enyi.
	\newblock On the strength of connectedness of a random graph.
	\newblock {\em Acta Mathematica. Academiae Scientiarum Hungaricae}, 12:261--267, 1961.
	\newblock \href {https://doi.org/10.1007/BF02066689} {\path{doi:10.1007/BF02066689}}.
	
	\bibitem{Frieze}
	A.~M. Frieze.
	\newblock On the independence number of random graphs.
	\newblock {\em Discrete Mathematics}, 81(2):171--175, 1990.
	\newblock \href {https://doi.org/10.1016/0012-365X(90)90149-C} {\path{doi:10.1016/0012-365X(90)90149-C}}.
	
	\bibitem{GGLS2021}
	Matteo Gallet, Georg Grasegger, Jan Legersk\'y, and Josef Schicho.
	\newblock On the existence of paradoxical motions of generically rigid graphs on the sphere.
	\newblock {\em SIAM Journal on Discrete Mathematics}, 35(1):325--361, 2021.
	\newblock \href {https://doi.org/10.1137/19M1289467} {\path{doi:10.1137/19M1289467}}.
	
	\bibitem{Gil59}
	E.~N. Gilbert.
	\newblock Random graphs.
	\newblock {\em Annals of Mathematical Statistics}, 30:1141--1144, 1959.
	\newblock \href {https://doi.org/10.1214/aoms/1177706098} {\path{doi:10.1214/aoms/1177706098}}.
	
	\bibitem{GIMPS}
	Ant\'onio Gir\~ao, Freddie Illingworth, Lukas Michel, Emil Powierski, and Alex Scott.
	\newblock Reconstructing a point set from a random subset of its pairwise distances.
	\newblock {\em SIAM Journal on Discrete Mathematics}, 38(4):2709--2720, 2024.
	\newblock \href {https://doi.org/10.1137/23M1586860} {\path{doi:10.1137/23M1586860}}.
	
	\bibitem{GHT10}
	Steven~J. Gortler, Alexander~D. Healy, and Dylan~P. Thurston.
	\newblock Characterizing generic global rigidity.
	\newblock {\em Amer. J. Math.}, 132(4):897--939, 2010.
	\newblock \href {https://doi.org/10.1353/ajm.0.0132} {\path{doi:10.1353/ajm.0.0132}}.
	
	\bibitem{GLS2019}
	Georg Grasegger, Jan Legersk\'y, and Josef Schicho.
	\newblock Graphs with flexible labelings.
	\newblock {\em Discrete \& Computational Geometry. An International Journal of Mathematics and Computer Science}, 62(2):461--480, 2019.
	\newblock \href {https://doi.org/10.1007/s00454-018-0026-9} {\path{doi:10.1007/s00454-018-0026-9}}.
	
	\bibitem{GLSinjective}
	Georg Grasegger, Jan Legersk\'y, and Josef Schicho.
	\newblock Graphs with flexible labelings allowing injective realizations.
	\newblock {\em Discrete Mathematics}, 343(6):111713, 14, 2020.
	\newblock \href {https://doi.org/10.1016/j.disc.2019.111713} {\path{doi:10.1016/j.disc.2019.111713}}.
	
	\bibitem{disco-cut}
	Takehiro Ito, Marcin Kami\'nski, Dani\"el Paulusma, and Dimitrios~M. Thilikos.
	\newblock On disconnected cuts and separators.
	\newblock {\em Discrete Appl. Math.}, 159(13):1345--1351, 2011.
	\newblock \href {https://doi.org/10.1016/j.dam.2011.04.027} {\path{doi:10.1016/j.dam.2011.04.027}}.
	
	\bibitem{JSS2007}
	Bill Jackson, Brigitte Servatius, and Herman Servatius.
	\newblock The 2-dimensional rigidity of certain families of graphs.
	\newblock {\em Journal of Graph Theory}, 54(2):154--166, 2007.
	\newblock \href {https://doi.org/10.1002/jgt.20196} {\path{doi:10.1002/jgt.20196}}.
	
	\bibitem{Kempe1877}
	A.~B. Kempe.
	\newblock On {C}onjugate {F}our-piece {L}inkages.
	\newblock {\em Proceedings of the London Mathematical Society}, 9:133--147, 1877/78.
	\newblock \href {https://doi.org/10.1112/plms/s1-9.1.133} {\path{doi:10.1112/plms/s1-9.1.133}}.
	
	\bibitem{Kriv97}
	Michael Krivelevich.
	\newblock Triangle factors in random graphs.
	\newblock {\em Combinatorics, Probability and Computing}, 6(3):337--347, 1997.
	\newblock \href {https://doi.org/10.1017/S0963548397003106} {\path{doi:10.1017/S0963548397003106}}.
	
	\bibitem{LNPR}
	Alan Lew, Eran Nevo, Yuval Peled, and Orit~E. Raz.
	\newblock Sharp threshold for rigidity of random graphs.
	\newblock {\em Bulletin of the London Mathematical Society}, 55(1):490--501, 2023.
	\newblock \href {https://doi.org/10.1112/blms.12740} {\path{doi:10.1112/blms.12740}}.
	
	\bibitem{MNS24}
	Richard Montgomery, Rajko Nenadov, and Tibor Szab\'{o}.
	\newblock Global {R}igidity of {R}andom {G}raphs in {${\mathbb{R}}$}.
	\newblock In {\em 2023 {MATRIX} {A}nnals}, volume~6 of {\em MATRIX Book Ser.}, pages 717--724. Springer, Cham, 2025.
	\newblock URL: \url{https://doi.org/10.1007/978-3-031-76738-8_47}, \href {https://doi.org/10.1007/978-3-031-76738-8\_47} {\path{doi:10.1007/978-3-031-76738-8\_47}}.
	
	\bibitem{Vil25}
	Soma Vill\'anyi.
	\newblock Every {$d(d+1)$}-connected graph is globally rigid in {$\mathbb R^d$}.
	\newblock {\em Journal of Combinatorial Theory. Series B}, 173:1--13, 2025.
	\newblock \href {https://doi.org/10.1016/j.jctb.2025.01.005} {\path{doi:10.1016/j.jctb.2025.01.005}}.
	
	\bibitem{whiteley97}
	Walter Whiteley.
	\newblock Rigidity and scene analysis.
	\newblock In {\em Handbook of discrete and computational geometry}, CRC Press Ser. Discrete Math. Appl., pages 893--916. CRC, Boca Raton, FL, 1997.
	
\end{thebibliography}
\end{document}